\documentclass[12pt,reqno,a4paper]{amsart}
\usepackage{blindtext}
\usepackage{fullpage}
\usepackage{mathtools}
\usepackage{longtable}
\usepackage{amsmath,amssymb,amsthm}
\usepackage{amscd}
\usepackage{bm}
\usepackage{hyperref}   
\usepackage{dsfont}
\usepackage{enumerate}
\usepackage{epsfig}
\usepackage{float, graphicx}
\usepackage{latexsym, amsxtra}
\usepackage{mathrsfs}
\usepackage{multicol}
\usepackage[normalem]{ulem}
\usepackage{psfrag}
\usepackage[parfill]{parskip}
\usepackage{stmaryrd}
\usepackage{tikz}
\usepackage[T1]{fontenc}
\usepackage{url}
\usepackage{verbatim}
\usepackage{indentfirst}
\usepackage{tikz-cd}
\usepackage{booktabs}
\usepackage{svg}

\usepackage{mathtools}

\makeatletter
\def\thm@space@setup{%
	\thm@preskip=2ex \thm@postskip=2ex
}
\makeatother

\hypersetup{hidelinks}

\newtheorem{thm}{Theorem~}[section]
\newtheorem{lem}[thm]{Lemma~}

\newtheorem{prop}[thm]{Proposition~}
\newtheorem{ques}[thm]{Question~}
\newtheorem{cor}[thm]{Corollary~}

\theoremstyle{remark}
\newtheorem{rmk}[thm]{Remark~}
\newtheorem{ex}[thm]{Example~}

\theoremstyle{definition}  
\newtheorem{defn}[thm]{Definition~}

\newcommand{\calA}{\mathcal{A}}

\newcommand{\CC}{\mathbb{C}}

\newcommand{\RR}{\mathbb{R}}
\newcommand{\LL}{\mathbb{L}}
\newcommand{\PP}{\mathbb{P}}

\newcommand{\HH}{\mathbb{H}}

\newcommand\mult{\mathrm{mult}}

\title{Homology of local systems on real line arrangement complements}

\author[B. Xie]{Baiting Xie}
\address{Qiuzhen College, Tsinghua University, China}
\email{xbt23@mails.tsinghua.edu.cn}

\author[C. Yu]{Chenglong Yu}
\address{Center for Mathematics and Interdisciplinary Sciences, Fudan University and
Shanghai Institute for Mathematics and Interdisciplinary Sciences (SIMIS), Shanghai, China}
\email{yuchenglong@simis.cn}
\date{}

\begin{document}
		
\begin{abstract}
	We study the homology groups of the complement of a complexified real line arrangement with coefficients in complex rank-one local systems. Using Borel--Moore homology, we establish an algorithm computing their dimensions via the real figures of the arrangement. It enables us to give a new upper bound. We further consider the case where the arrangement contains a sharp pair and make partial progress on a conjecture proposed by Yoshinaga.
\end{abstract}
	
	\maketitle
     \setcounter{tocdepth}{1}

	\tableofcontents

\section{Introduction}\label{section: introduction}
In the theory of hyperplane arrangements, a classical problem is to determine which topological invariants of their complements are governed by their combinatorics. The pioneering work \cite{MR558866} of Orlik and Solomon proves that the cohomology ring is combinatorially determined. On the other hand, Rybnikov's counter-example \cite{rybnikov2011fundamental} (see also \cite{MR2188450}) implies that the fundamental group is not. An important intermediate case between them is homology groups with coefficients in complex rank-one local systems. This is also related to the monodromy on cohomology of the corresponding Milnor fibration.

More explicitly, let $ \calA $ be a hyperplane arrangement in $ \CC\PP^{n} $ and  $ M(\calA) $ be its complement. The combinatorial data of $ \calA $ are encoded in its intersection lattice $ L(\calA) $. This lattice consists of all intersections of hyperplanes in $ \calA $, ordered by reverse inclusion and ranked by codimension. For any given field $ \mathbb{K} $, taking the counter-clockwise monodromy around each hyperplane in $ \calA $ yields the following one-to-one correspondence:
\begin{equation*}
    \left\{
            \begin{tabular}{c}
                isomorphism classes of rank-one \\
                  $ \mathbb{K} $-local systems $ \LL $ on $ M(\calA) $
            \end{tabular}
        \right\}
        \quad\stackrel{1:1}{\longleftrightarrow}\quad
        \left\{
            \begin{tabular}{c}
                 $ m \colon \calA \rightarrow \mathbb{K}^{\times} $ such that \\
                 $ \prod\limits_{H \in \calA} m(H) = 1 $
            \end{tabular}
        \right\}.
\end{equation*}

 The problem mentioned above can then be precisely stated as follows, which has been intensively studied but still remains widely open.
\begin{ques}
	Given a hyperplane arrangement $ \calA $ and a rank-one $ \mathbb{K} $-local system $ \LL $ on $ M(\calA) $, are the Betti numbers  $ h_{*}(M(\calA),\LL) = \dim_{\mathbb{K}} H_{*}(M(\calA),\LL) $ determined by the intersection lattice $ L(\calA) $ and the monodromy map $ m \colon \calA \rightarrow \mathbb{K}^{\times} $?
\end{ques}

In the sequel, we focus on the case where $ \calA $ is a line arrangement in $ \CC\PP^{2} $. In this setting, the intersection lattice $ L(\calA) $ can be recovered from the set $ L_{2}(\calA) $ of all intersection points and the associated sets $ \calA_{p} $ of lines in $ \calA $ containing each point $ p \in L_{2}(\calA) $. Let $ \mult(p) $ denote the size of $ \calA_{p} $. Note that $ h_{0}(M(\calA),\LL) = 0  $ when $ \LL $ is nontrivial, and the Euler characteristic of $ \LL $ is determined by $ L(\calA) $. Thus it suffices to compute $h_{1}(M(\calA),\LL)$. When $\mult(p)\in \{2,3\}$ for all $ p \in L_{2}(\calA) $, the work of Papadima--Suciu \cite{papadima2017milnor} shows how $h_{1}(M(\calA),\LL)$ is determined by $L(\calA)$ when $\LL$ is the local system arising from Milnor fibers of $\calA$.

When $ \calA $ is a complexified real line arrangement and $ \LL $ is a complex rank-one local system on $ M(\calA) $, the planar graph given by the real locus of $ \calA $ provides a foundation for several approaches to compute $ h_{1}(M(\calA),\LL) $. The first effective topological algorithm was proposed by Cohen and Suciu \cite{MR1310725}. They compute $ h_{1}(M(\calA),\LL) $ by applying Fox calculus to Randell's presentation of $ \pi_{1}(M(\calA)) $ \cite{MR671654}. Using twisted minimal chain complexes, Yoshinaga \cite{MR3090727} developed a method for computing $ h^{1}(M(\calA),\LL) $ that does not involve $ \pi_{1}(M(\calA)) $. However, both algorithms rely on the real figure, which contains more data than the intersection lattice $ L(\calA) $.

Less ambitiously, one may seek for a combinatorial upper bound on $ h_{1}(M(\calA),\LL) $. The concept of resonant points arises naturally in the study of this problem. A point $ p \in L_{2}(\calA) $ is called \textbf{resonant} (with respect to $\LL$) if $ \mult(p) \geq 3$ and $ \prod_{l \in \calA_{p}} m(l) = 1$. Cohen, Dimca and Orlik \cite{MR2038782} proved the following theorem, which provides the best known upper bound on $ h_{1}(M(\calA),\LL) $ for a general line arrangement $ \calA $. Note that although they only proved the case of constant monodromy, their method also holds for general complex local systems $ \LL $.
\begin{thm}[{\cite[Theorem 13]{MR2038782}}]
	Let $ \calA $ be a line arrangement in $ \CC\PP^{2} $.Let $ \LL $ be a complex rank-one local system on $ M(\calA) $ with monodromy map $ m \colon \calA \rightarrow \CC^{\times} $ satisfying that $ m(l) \neq 1 $ for each $ l \in \calA $. Let $ l_{0} $ be a line in $ \calA $ and denote by $ R_{0} $ the set of resonant points on $ l_{0} $. Then
	\begin{equation*}
		h_{1}(M(\calA),\LL) \leq \sum\limits_{p \in R_{0}} (\mult(p)-2).
	\end{equation*}
\end{thm}

In this paper, we focus on the case where $ \calA $ is a complexified real line arrangement and $ \LL $ is a rank-one $ \mathbb{K} $-local system such that the monodromy $ m(l) \neq 1 $ for each $ l \in \calA $. Using Borel--Moore homology, we establish a method for computing $ h_{1}(M(\calA),\LL) $, which can be regarded as a dual version of Yoshinaga's approach. This method enables a more detailed combinatorial analysis. As a result, we prove the following theorem.
\begin{thm}\label{thm: main}
	Let $ \calA $ be a complexifed real line arrangement in $ \CC\PP^{2} $. Let $ \mathbb{K} $ be a field. Let $ \LL $ be a rank-one $ \mathbb{K} $-local system on $ M(\calA) $ with monodromy map $ m \colon \calA \rightarrow \mathbb{K}^{\times} $ satisfying that $ m(l) \neq 1 $ for each $ l \in \calA $. Let $ l_{0} $ be a line in $ \calA $ and denote by $ R_{0} $ the set of resonant points on $ l_{0} $. If $ \# L_{2}(\calA) > 1 $, then we have
	\begin{equation*}
		h_{1}(M(\calA),\LL) \leq \max(0,\# R_{0}-1).
	\end{equation*}
\end{thm}

\begin{rmk}
	When $ \# L_{2}(\calA) = 1 $, all lines in $ \calA $ pass through a common point. Then $ M(\calA) $ deformation retracts onto the $ \CC\PP^{1} $ with $ \# \calA $ points removed. As shown in Example \ref{example: punctured P1}, $ 	h_{1}(M(\calA),\LL) = \# \calA-2 $.
\end{rmk}

As an application, we consider the case where $ \calA $ contains a sharp pair, which was first considered by Yoshinaga \cite{MR3090727}. Note that each pair of lines in $ \calA $ divides $ \RR\PP^{2} $ into two connected components. Such a pair is called a \textbf{sharp pair} if one of these two components contains no intersection points in $ L_{2}(\calA) $. The following theorem is a slight generalization of Yoshinaga's theorem 3.21 in \cite{MR3090727}.
\begin{thm}\label{thm: main for sharp pair}
		Let $ \calA $ be a complexified real line arrangement in $ \CC\PP^{2} $. Let $ \mathbb{K} $ be a field. Let $ \LL $ be a rank-one $ \mathbb{K} $-local system on $ M(\calA) $ with monodromy map $ m \colon \calA \rightarrow \mathbb{K}^{\times} $ satisfying that $ m(l) \neq 1 $ for each $ l \in \calA $. If $ \calA $ contains a sharp pair, then $ h_{1}(M(\calA),\LL) \leq 1 $.
\end{thm}

In particular, for the case where $ m $ is a constant map, the following theorem is a partial progress towards Conjecture 1.2 in \cite{MR3666711}.
\begin{thm}\label{thm: odd monodromy}
	 Let $ \calA $ be a complexified real line arrangement in $ \CC\PP^{2} $. Let $ \mathbb{K} $ be a field. Let $ \LL $ be a rank-one $ \mathbb{K} $-local system on $ M(\calA) $ with constant monodromy $ \zeta \in \mathbb{K}^{\times} $. If $ \calA $ contains a sharp pair and $ H_{1}(M(\calA),\LL) \neq 0 $, then there exists an odd number $ d $ such that $ \zeta^{d} = 1 $.
\end{thm}

The structure of the paper is as follows. In \S\ref{section: preparation}, we review some basic facts on Borel--Moore homology. Using these preliminary results, an algorithm is established for computing $ h_{1}(M(\calA),\LL) $ in \S\ref{section: main construction}. In \S\ref{section: proof of main theorem} we prove Theorem \ref{thm: main}. The case where $ \calA $ contains a sharp pair is discussed in \S\ref{section: sharp pair}.

\textbf{Acknowledgement.}
The second author is supported by the national key research and development program of China (No. 2022YFA1007100) and NSFC 12201337. We thank Qiliang Luo and Alex Suciu for their interest and helpful discussions. We also thank Yongqiang Liu and Laren\c{t}iu Maxim for their comments on the intial version of this paper. Based on their suggestions, we have generalized the coefficient field from $ \CC $ to an arbitrary field $ \mathbb{K} $.

\section{Preparation on homology of local systems}\label{section: preparation}

\subsection{Borel--Moore homology with local coefficients}\label{subsection: BM homology}
First we recall some basic facts and notation about Borel--Moore homology $ H^{BM}_{i}(X,\LL) $ of connected topological manifolds $ X $ with coefficients in rank-one $ \mathbb{K} $-local systems $ \LL $. For more details, see \cite{MR1481706}.

Denote by $ C^{BM}_{i}(X,\LL) $ the space of locally finite chains on $ X $ with coefficients in $ \LL $. Endowing them with the usual boundary maps $ \partial_{i} \colon C^{BM}_{i}(X,\LL) \rightarrow C^{BM}_{i-1}(X,\LL) $, we obtain a chain complex:

\begin{equation*}
	\cdots \stackrel{\partial_{3}}{\longrightarrow} C^{BM}_{2}(X,\LL) \stackrel{\partial_{2}}{\longrightarrow} C^{BM}_{1}(X,\LL) \stackrel{\partial_{1}}{\longrightarrow} C^{BM}_{0}(X,\LL) \longrightarrow 0.
\end{equation*}

The Borel--Moore homology $ H^{BM}_{*}(X,\LL) $ of $ X $ with coefficients in $ \LL $ is defined to be the homology groups of the above chain complex. By definition, the space of singular chains $ C_{*}(X,\LL) $ can be embedded into $ C^{BM}_{*}(X,\LL) $, which induces a natural map $ H_{*}(X,\LL) \rightarrow H^{BM}_{*}(X,\LL)  $.

\begin{ex}\label{example: BM of n-cells}
	Let $ \Delta $ be an oriented open $ n $-cell and $ \LL $ a rank-one $ \mathbb{K} $-local system on $ \Delta $. Since $ \Delta $ is contractible, $ \LL $ is isomorphic to the constant sheaf $ \underline{\mathbb{K}} $ and there exists a nonzero global section $ e $ of $ \LL $. Then
	\begin{equation*}
		H^{BM}_{i}(\Delta,\LL) = \begin{cases}
			\mathbb{K} & i=n\\
			0 & \text{otherwise}
		\end{cases}
	\end{equation*}
	and the generator of $ H^{BM}_{n}(\Delta, \LL)$ is given by a locally finite chain $ \Delta \otimes e $. 
\end{ex}

Now we introduce some basic properties of Borel--Moore homology. The following lemma provides a useful long exact sequence.

\begin{lem}[{\cite[Corollary V.5.10]{MR1481706}}]\label{lem: localization sequence of BM}
	Let $ F $ be a closed subset of a locally compact space $ X $ with $ U = X - F $ its complement. Then there exist push-forward maps $  H^{BM}_{*}(F,\LL) \rightarrow  H^{BM}_{*}(X,\LL) $, restriction homomorphisms $ H^{BM}_{*}(X,\LL) \rightarrow  H^{BM}_{*}(U,\LL) $, and a long exact localization sequence:
	\begin{equation*}
		\cdots \longrightarrow H^{BM}_{i}(F,\LL)\longrightarrow H^{BM}_{i}(X,\LL) \longrightarrow H^{BM}_{i}(U,\LL) \longrightarrow H^{BM}_{i-1}(F,\LL) \longrightarrow \cdots.
	\end{equation*}
\end{lem}

We remark that the connecting homorphisms $ H^{BM}_{i}(U,\LL) \longrightarrow H^{BM}_{i-1}(F,\LL) $ in the above long exact sequence are called the residue maps.

The following lemma is a corollary of generalized Poincar\'e duality.
\begin{lem}[{\cite[Theorem V.9.2]{MR1481706}}]\label{lem: duality of BM}
	Let $ X $ be a connected oriented smooth manifold of dimension $ n $. Then
	\begin{equation*}
		H^{BM}_{i}(X,\LL) \simeq H^{n-i}(X,\LL),
	\end{equation*}
	where $ H^{i}(X,\LL) $ is the $ i $-th singular cohomology on $ X $ of coefficients in $ \LL $.
\end{lem}

The following lemma compares Borel--Moore homology with singular homology when $ X $ admits a good compactification.
\begin{lem}\label{lem: compare BM with singular}
	Let $ X $ be a smooth connected complex manifold. Assume that there exists a compactification $ \widetilde{X} $ of $ X $ such that 
	\begin{enumerate}
		\item $ \widetilde{X} $ is smooth and complete,
		\item $ D = \widetilde{X}-X $ is a normal crossing divisor with smooth irreducible components, and
		\item the monodromy of $ \LL $ around each component of $ D $ is non-trivial.
	\end{enumerate}
	Then the natural maps
	\begin{equation*}
		H_{*}(X,\LL) \stackrel{\simeq}{\longrightarrow} H^{BM}_{*}(X,\LL)
	\end{equation*}
	are isomorphisms.
\end{lem}

\begin{proof}
	The assumptions ensures that $ Rj_{!}\LL \simeq Rj_{*}\LL $, where $ j \colon X \hookrightarrow \widetilde{X} $ is the open embedding. Then we have $ H_{c}^{*}(X,\LL) \simeq \HH^{*}(\widetilde{X},Rj_{!}\LL) \simeq \HH^{*}(\widetilde{X},Rj_{*}\LL) \simeq H^{*}(X,\LL) $. So by Lemma \ref{lem: duality of BM} we have $ H_{*}(X,\LL) \simeq H^{BM}_{*}(X,\LL) $.
\end{proof}

\begin{ex}\label{example: punctured P1}
	Let $ F \subset \CC\PP^{1} $  be a finite set of points in $ \CC\PP^{1} $. Let $ \LL $ be a nontrivial rank-one $ \mathbb{K} $-local system on $ E^{\circ} = \CC\PP^{1}\setminus F $ with monodromy map $ m \colon F \rightarrow \mathbb{K}^{\times} $. Then we can compute $ H^{BM}_{*}(E^{\circ},\LL) $ as follows.
	
	Up to a homeomorphism, we may assume that $ F $ is contained in $ \RR \subset \CC\PP^{1} $. Suppose that $ F = \{p_{1},\cdots,p_{k}\} $, where $ p_{1} < p_{2} < \cdots < p_{k} $. Denote by $ \sigma_{i} $ the open segment $ (p_{i},p_{i+1}) $, where $ p_{k+1} = p_{1} $ and $ \sigma_{k} $ contain the infinity point. Denote by $ \Delta^{+} $ (resp. $ \Delta^{-} $) the open $ 2 $-cell given by the upper (resp. lower) half plane in $ E^{\circ} $, endowed with the clockwise orientation. 
	
	Choose a non-trivial global section $ e^{+} $ of $ \LL $ on the closure $ \overline{\Delta^{+}} $. Extend $ e^{+}|_{[0:1]} $ from $ [0:1] $ to a global section of $ \LL $ on $ \overline{\Delta^{-}} $, denoted by $ e^{-} $. Then for any $ 1 \leq i \leq k $, we have
 	\begin{equation*}
 		e^{-}|_{\sigma_{i}} = m(p_{1})\cdots m(p_{i})e^{+}|_{\sigma_{i}}.
 	\end{equation*}
 	
 	By Lemma \ref{lem: duality of BM}, we have $ \dim H^{BM}_{2}(E^{\circ}, \LL) = \dim H^{0}(E^{\circ}, \LL) = 0 $ and $ \dim H^{BM}_{0}(E^{\circ}, \LL) = \dim H^{2}(E^{\circ}, \LL) = 0 $. And by Lemma \ref{lem: localization sequence of BM} we have the following short exact sequence
 	\begin{equation*}
 		0 \longrightarrow H^{BM}_{2}(\Delta^{+},\LL) \oplus H^{BM}_{2}(\Delta^{-},\LL) \longrightarrow \bigoplus\limits_{i=1}^{k} H^{BM}_{1}(\sigma_{i},\LL) \longrightarrow H^{BM}_{1}(E^{\circ}, \LL) \longrightarrow 0.
 	\end{equation*}
 	
 	Combining with Example \ref{example: BM of n-cells}, we have $ [\sigma_{i}\otimes e^{+}] (1 \leq i \leq k) $ span $ H^{BM}_{1}(E^{\circ}, \LL) $, with all relations among them given by 
 	\begin{equation*}
 		\partial[\Delta^{+}\otimes e^{+}] = -\sum\limits_{i=1}^{k}[\sigma_{i}\otimes e^{+}]
 	\end{equation*}
	and
	\begin{equation*}
		\partial[\Delta^{-}\otimes e^{-}] = \sum\limits_{i=1}^{k}[\sigma_{i}\otimes e^{-}] = \sum\limits_{i=1}^{k}m(p_{1})\cdots m(p_{i})[\sigma_{i}\otimes e^{+}].
	\end{equation*} 	
\end{ex}

\subsection{Computing the first homology of line arrangement complements}\label{subsection: compute H1}

Given a line arrangement $ \calA $, let $ \pi \colon Y \rightarrow \CC\PP^{2} $ be the blowing-up at all points $ p \in L_{2}(\calA) $ with $ \mult(p) > 2 $. For such points $ p $, let $ E_{p} $ denote the exceptional divisor at $ p $. Denote by $ E_{p}^{\circ} = E_{p} - \bigcup_{l \in \calA_{p}} l $, where $ \calA_{p} $ is the set of lines in $ \calA $ containing the point $ p $. Then $ \pi^{-1}\calA $ defines a normal crossing divisor on $ Y $ whose irreducible components are the strict transforms of lines in $ \calA $ and the exceptional divisors $ E_p $. The complement of this divisor in $ Y $ is still $ M(\calA) $.

Let $ \LL $ be a rank-one $ \mathbb{K} $-local system on $ M(\calA) $ with monodromy map $ m \colon \calA \rightarrow \mathbb{K}^{\times} $ satisfying that $ m(l) \neq 1 $ for each $ l \in \calA $. Then the monodromy of $ \LL $ around $ E_{p} $ is exactly $ \prod_{l \in \calA_{p}} m(l) $. Let $ R $ denote the set of resonant points with respect to $ \LL $. By definition, the monodromy of $ \LL $ around $ E_{p} $ is trivial if and only if $ p \in R $. Thus the local system $ \LL $ can be extended to one on the open set 
\begin{equation*}
    \widetilde{M} =  M(\calA) \cup  (\bigcup\limits_{p \in R}E_{p}^{\circ}),
\end{equation*}
which we still denote by $ \LL $. 

The following lemma implies that the first cohomology remains invariant under this extension.

\begin{lem}\label{lem: H1 remains invariant under extension}
    The inclusion map induce an isomorphism
    \begin{equation*}
        H_{1}(M(\calA),\LL) \stackrel{\simeq}{\longrightarrow} H_{1}(\widetilde{M},\LL).
    \end{equation*}
\end{lem}
\begin{proof}
    Since $ \bigsqcup_{p \in R}E_{p}^{\circ} $ is a smooth submanifold of $ \widetilde{M} $, we have the following Gysin sequence:
	\begin{equation*}
		0 \longrightarrow H^{1}(\widetilde{M}, \LL^{\vee}) \longrightarrow H^{1}(M(\calA), \LL^{\vee}) \longrightarrow \bigoplus\limits_{p \in R}  H^{0}(E_{p}^{\circ},\LL^{\vee}) \longrightarrow H^{2}(\widetilde{M}, \LL^{\vee}).
	\end{equation*}
	For each $ p \in R $, since the restriction of $ \LL^{\vee} $ on $ E_{p}^{\circ} $ is nontrivial, we have $ H^{0}(E_{p}^{\circ},\LL^{\vee}) = 0 $. So $ H^{1}(\widetilde{M}, \LL^{\vee}) \simeq H^{1}(M(\calA), \LL^{\vee}) $. Thus we have $ H_{1}(M(\calA),\LL) \simeq H_{1}(\widetilde{M},\LL) $.
\end{proof}
\begin{rmk}
    The same method can be used to prove a similar statement in a more general context. Let $ X $ be a compact smooth complex surface with a simple normal crossing divisor $ D $ with smooth components. Denote by $ U = X \setminus D $ the complement of the divisor $ D $. Let $ \LL $ be a rank-one $ \mathbb{K} $-local system on $ U $. Decompose the divisor $ D $ into $ D_{0}+E $, where $ D_{0} $ is the sum of the components of $ D $ with nontrivial monodromy of $ \LL $, and $ E $ consists of those with trivial monodromy. Then the local system $ \LL $ can be extended to one on the open set $ \widetilde{U} = 
    X \setminus D_{0} $, which we still denote by $ \LL $. Assume that each component of $ E $ has nonempty intersection with $ D_{0} $. By successively adding the components of $ E $ to $ U $, one can prove that $ H_{1}(U,\LL) \simeq H_{1}(\widetilde{U},\LL) $.
\end{rmk}

The following key lemma provides a description of $ H_{1}(M(\calA),\LL)  $.   

\begin{lem}\label{lem: H1 is given by exceptional divisors}
	There exists a surjective $ \mathbb{K} $-linear map 
	\begin{equation*}
		\bigoplus\limits_{p \in R}H^{BM}_{1}(E_{p}^{\circ},\LL) \twoheadrightarrow  H_{1}(M(\calA),\LL)  
	\end{equation*}
whose kernel is given by the image of the residue map
	\begin{equation*}
		 H^{BM}_{2}(M(\calA),\LL)\longrightarrow \bigoplus\limits_{p \in R}  H^{BM}_{1}(E_{p}^{\circ},\LL).
	\end{equation*}
\end{lem}

\begin{proof}
	
	By Lemma \ref{lem: localization sequence of BM} we have the following long exact sequence:
	\begin{equation*}
		H^{BM}_{2}(M(\calA),\LL)\longrightarrow \bigoplus\limits_{p \in R}  H^{BM}_{1}(E_{p}^{\circ},\LL) \longrightarrow H^{BM}_{1}(\widetilde{M},\LL) \longrightarrow H^{BM}_{1}(M(\calA),\LL).
	\end{equation*}
	
	Note that the monodromy of $ \LL $ around each component of $ Y \setminus \widetilde{M} $ is non-trivial. So by Lemma \ref{lem: compare BM with singular} and \ref{lem: H1 remains invariant under extension}, we have $ H^{BM}_{1}(\widetilde{M},\LL) \simeq H_{1}(\widetilde{M},\LL) \simeq H_{1}(M(\calA),\LL) $.  Furthermore, by Lemma \ref{lem: duality of BM}, $ \dim H^{BM}_{1}( M(\calA),\LL) = \dim H^{3}(M(\calA),\LL) = 0 $. Therefore, the exact sequence above yields a surjective map
	\begin{equation*}
		\bigoplus\limits_{p \in R}H^{BM}_{1}(E_{p}^{\circ},\LL) \twoheadrightarrow  H_{1}(M(\calA),\LL)  
	\end{equation*}
whose kernel is the image of the residue map
	\begin{equation*}
		 H^{BM}_{2}(M(\calA),\LL)\longrightarrow \bigoplus\limits_{p \in R}  H^{BM}_{1}(E_{p}^{\circ},\LL).
	\end{equation*}
\end{proof}

\section{An algorithm for real line arrangements}\label{section: main construction}

In this section, we consider the case where $ \calA $ is a complexified real line arrangement. In this case, the real figure of $ \calA $ allows us to describe the topology of $ M(\calA) $ in a more combinatorial way. Using Lemma \ref{lem: H1 is given by exceptional divisors}, we generate $ H_{1}(M(\calA),\LL) $ via angles at resonant points and describe the relations among them using bounded chambers.

\subsection{Settings and convention}\label{subsection: setting}
Fixing homogeneous coordinates $ [x:y:z] $ on $ \mathbb{CP}^2 $, let $ \calA $ be a complexified real line arrangement in $ \CC\PP^{2} $; that is, every line in $ \calA $ can be defined by a real linear form.  We keep the notation $\LL,m,R,E_{p},E_{p}^{\circ},\widetilde{M} $ from subsection \ref{subsection: compute H1}.

Let $ l_{\infty} $ be a generic real line in $ \CC\PP^{2} $; that is, $ l_{\infty} $ is defined by a real linear form, with $ l_{\infty} \cap L_{2}(\calA) = \emptyset $. Then, by taking $ l_{\infty} $ as the line at infinity $ \{z=0\} $ and applying a further real coordinate transformation, we may assume that the slopes of the lines in $ \calA $ are distinct finite numbers. Identify the affine chart $ \{[x:y:1] \mid x,y \in \CC\} $ with $ \CC^{2} $, equipped with the coordinates $ x,y $. Then each $ l \in \calA $ can be defined by an equation
\begin{equation*}
	Q_{l}(x,y) = y -s(l)x - b(l) = 0,
\end{equation*}
where the slopes $ s(l) \in \RR $ are distinct and $ b(l) \in \RR $. 

 Denote by
\begin{equation*}
    \overline{\HH^{+}} = \{p \in \CC^{2} \mid x(p) \in \RR,\ \operatorname{Im}(y(p)) > 0\} \cup (\RR^{2}\cap M(\calA)) \subset M(\calA).
\end{equation*}
 
 Over the simply connected set $ \overline{\HH^{+}} $, fix a nonzero section $ e $ of $ \LL $ as a global marking. For each $ p \in R $, note that the exceptional divisor $ E_{p} $ is isomorphic to $ \CC\PP^{1} $, equipped with the homogeneous coordinates $ [x:y] $ inherited from $ \CC\PP^{2} $.  Under this identification, we have
\begin{equation*}
	E_{p}^{\circ} = \CC\PP^{1} -  \{[1:s(l)] \mid l \in \calA_{p}\}.
\end{equation*}
We further denote by
\begin{equation*}
    \overline{\HH_{p}^{+}} = \{[x:y] \in E_{p}^{\circ} \mid \operatorname{Im}(\bar{x}y) \geq 0\} \subset E_{p}^{\circ}
\end{equation*}
and extend $ e $ to a global section of $ \LL $ over the simply connected set $ \overline{\HH_{p}^{+}} $ as follows. First extend $ e $ along the segment 
\begin{equation*}
	\gamma_{p} \colon [0,1] \rightarrow \widetilde{M},\ \gamma_{p}(t) = (x(p),y(p)+(1-t)i) \in \overline{\HH^{+}},\ \forall t \in [0,1);\ \gamma_{p}(1) = [0:1] \in \overline{\HH_{p}^{+}}.
\end{equation*}
Then extend it further from $ [0:1] $ to a global section over $ \overline{\HH_{p}^{+}} $, denoted by $ e_{p} $.

\subsection{Generators given by angles and Relations given by bounded chambers}\label{subsection: generators and relations}	
	For each $ p \in R $, the space $ H^{BM}_{1}(E_{p}^{\circ},\LL) $ can be described in terms of angles at $ p $. These angles are defined as follows.
	\begin{defn}\label{def: angles}
		 Given a resonant point $ p \in R $, recall that $ E_{p} \simeq \CC\PP^{1} $ is equipped with the coordinates $ [x:y] $ inherited from $ \CC\PP^{2} $. An angle at $ p $ is defined to be a connected component of $ \RR\PP^{1} \cap E_{p}^{\circ} $. Define $ A(p) $ to be the $ \mathbb{K} $-vector space spanned by angles at $ p $. Let $ A = \bigoplus_{p \in R} A(p) $ be the space spanned by all angles of $ \calA $.
	\end{defn}

	More explicitly, suppose that $ \calA_{p} = \{ l_{1},\cdots,l_{k} \}$, where $ s(l_{1}) < \cdots < s(l_{k}) $. Then each angle $ \alpha $ at $ p $ is given by an open segment in $ \RR\PP^{1} $ with endpoints $ [1:s(l_{i})] $ and $ [1:s(l_{i+1})] $ for some $ i $, where $ l_{k+1} = l_{1} $. We represent this angle by the pair $ (l_{i},l_{i+1}) $ and endow it with an orientation from  $ [1:s(l_{i})] $ to $ [1:s(l_{i+1})] $. See Figure \ref{fig:1}.

    \begin{figure}[H]
 	\centering
 	\includegraphics[width=0.7\linewidth]{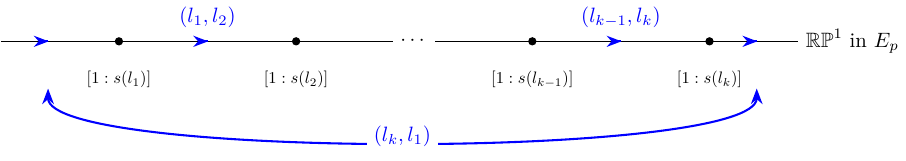}
 	\caption[1]{The angles at $ p $}
 	\label{fig:1}
 \end{figure}
	
	Recall that we have defined a nonzero global section $ e_{p} $ of $ \LL $ over $ \overline{\HH_{p}^{+}} \supset \RR\PP^{1} \cap E_{p}^{\circ} $. Then for each angle $ \alpha $, $ \alpha \otimes e_{p} $ defines a locally finite $ 1 $-cycle on $ E_{p}^{\circ} $ with coefficients in $ \LL $, which yields a $ \mathbb{K} $-linear map
	\begin{equation*}
		\begin{array}{rccc}
			f_{p} \colon & A(p) & \longrightarrow & H^{BM}_{1}(E_{p}^{\circ},\LL) \\
			& \alpha & \mapsto & [\alpha \otimes e_{p}]
		\end{array}
	\end{equation*}

	The following lemma is a direct corollary of Example \ref{example: punctured P1}. 
	\begin{lem}\label{lem: generators for CP1}
		The $ \mathbb{K} $-linear map $ f_{p} $ is surjective, whose kernel is spanned by 
		\begin{equation*}
			\alpha(p)^{+} =	\sum\limits_{i=1}^{k} (l_{i},l_{i+1})
		\end{equation*}
		and 
		\begin{equation*}
			\alpha(p)^{-} = \sum\limits_{i=1}^{k} m(l_{1})\cdots m(l_{i})(l_{i},l_{i+1}).
		\end{equation*}
	\end{lem}
	
	By collecting all $ f_{p} $ together, we obtain a surjective map
    \begin{equation*}
		f \colon A \twoheadrightarrow \bigoplus\limits_{p \in R}H^{BM}_{1}(E_{p}^{\circ},\LL).
	\end{equation*} 
    
    Composing $ f $ with the map in Lemma \ref{lem: H1 is given by exceptional divisors} yields a surjective map
	\begin{equation}\label{eqn: definition of F}
		F \colon A \twoheadrightarrow H_{1}( M(\calA),\LL).
	\end{equation} 

     By Lemma \ref{lem: H1 is given by exceptional divisors}, The kernel $ K $ of $ F $ arises from two parts. One part is the kernel of each $ f_{p} $, and the other part comes from the image of $ H^{BM}_{2}(M(\calA),\LL) $ under the residue map. To describe the space $ H^{BM}_{2}(M(\calA),\LL) $, we introduce the notion of chambers.
\begin{defn}\label{def: bounded chambers}
	  A chamber of $ \calA $ is a connected component of $ \RR^{2} \cap M(\calA) $. A chamber $ \Delta $ is called bounded if its area is finite. Let $ C $ denote the set consisting of all bounded chambers.
\end{defn}

By definition, each $ \Delta \in C $ is a closed subset of $ M(\calA) $ homeomorphic to an open $ 2 $-cell. Recall that we have fixed a nonzero global section $ e $ of $ \LL $ over $ \overline{\HH^{+}} \supset \Delta $. So with $ \Delta $ oriented clockwise, $ \Delta \otimes e $ defines a Borel--Moore homology class $ [\Delta \otimes e] $ in $ H^{BM}_{2}(M(\calA),\LL) $. The following lemma shows that these elements span the whole space. 
\begin{lem}\label{lem: generator of H_2}
	 Under the assumptions in subsection \ref{subsection: setting}, the space $ H^{BM}_{2}(M(\calA),\LL) $ is spanned by $ [\Delta \otimes e](\Delta \in C) $. 
\end{lem}

\begin{proof}
	We need the following well-known lemma.
	\begin{lem}\label{lem: deleting unbounded regions}
		Under the assumptions in subsection \ref{subsection: setting}, the open set $  M_{0} =  M(\calA)\setminus (\bigcup_{\Delta \in C}\Delta) $ deformation retracts onto $ l_{\infty}^{\circ} = l_{\infty} \cap  M(\calA) $.
	\end{lem}
	
	Then we have
	\begin{equation*}
		H^{2}( M_{0},\LL^{\vee}) = H^{2}(l_{\infty}^{\circ},\LL^{\vee}) = 0. 
	\end{equation*}
	
	By Lemma \ref{lem: duality of BM}, we have $ \dim H^{BM}_{2}(M_{0},\LL) = \dim H^{2}(M_{0},\LL^{\vee}) = 0 $. So by Lemma \ref{lem: localization sequence of BM} we have the following exact sequence:
		\begin{equation*}
	\bigoplus\limits_{\Delta \in C}H^{BM}_{2}(\Delta,\LL)\longrightarrow H^{BM}_{2}(M(\calA),\LL) \longrightarrow H^{BM}_{2}(M_{0},\LL) = 0.
\end{equation*}

As shown in Example \ref{example: BM of n-cells}, for each $ \Delta \in C $, the image of $ H^{BM}_{2}(\Delta,\LL) \rightarrow H^{BM}_{2}(M(\calA),\LL) $ is spanned by $ [\Delta \otimes e] $. Thus the surjectivity of the map $ \bigoplus_{\Delta \in C}H^{BM}_{2}(\Delta,\LL) \twoheadrightarrow H^{BM}_{2}(M(\calA),\LL)  $ implies the conclusion.
\end{proof}

For each chamber $ \Delta $, denote by $ V(\Delta) $ the set of its resonant vertices and $ \overline{\Delta} $ its closure of in $ \widetilde{M} $. Then for any $ p \in V(\Delta) $, the intersection $ \overline{\Delta} \cap E_{p}^{\circ} $ forms an angle at $ p $, which we denote by $ \alpha_{p}(\Delta) $. See Figure \ref{fig:2}

 \begin{figure}[H]
 	\centering
 	\includegraphics[width=0.7\linewidth]{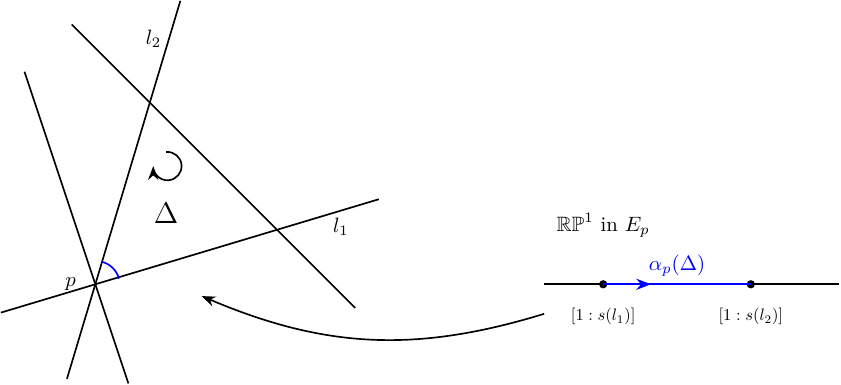}
 	\caption[1]{Definition of $ \alpha_{p}(\Delta) $}
 	\label{fig:2}
 \end{figure}

Denote by $ e(\overline{\Delta}) $ the extension to $ \overline{\Delta} $ of the section $ e|_{\Delta} $ over $ \Delta $. For each $ p \in V(\Delta) $, since $ \alpha_{p}(\Delta) $ is simply connected, there exists a complex number $ \lambda_{p}(\Delta) $ such that $ e(\overline{\Delta}) = \lambda_{p}(\Delta)e_{p} $ on $ \alpha_{p}(\Delta) $. When $ \Delta \in C $ is bounded, the following lemma describe the image of $ [\Delta \otimes e] $ under the residue map in terms of $ \lambda_{p}(\Delta) $ and $ \alpha_{p}(\Delta) $.  
	\begin{lem}\label{lem: residue of each bounded chamber}
    For a bounded chamber $ \Delta \in C $, define
    	\begin{equation*}
			\alpha(\Delta) = \sum\limits_{p \in V(\Delta)} \lambda_{p}(\Delta)\alpha_{p}(\Delta).
		\end{equation*}
        
		 Then the image of $ [\Delta \otimes e] $ under the residue map is $ f(\alpha(\Delta)) $. In particular, $ \alpha(\Delta) \in K $.
	\end{lem}
	
\begin{proof}

		Since $ \Delta $ is bounded, the closure of $ \Delta $ in $ \widetilde{M} $ is exactly
		\begin{equation*}
			\overline{\Delta} = \Delta \cup (\bigcup\limits_{p \in V(\Delta)} \alpha_{p}(\Delta)),
		\end{equation*}
		 where the orientation of $ \Delta $ matches those of $ \alpha_{p}(\Delta) $ on the boundary. 
		 
		 By definition, the image of $ [\Delta \otimes e] $ under the residue map can be represented by the boundary of $ \overline{\Delta}\otimes e(\overline{\Delta}) $, which equals to
	\begin{equation*}
		[\partial(\overline{\Delta}\otimes e(\overline{\Delta}))] = \sum\limits_{p \in V(\Delta)} [\alpha_{p}(\Delta) \otimes 	e(\overline{\Delta})] = \sum\limits_{p \in V(\Delta)} \lambda_{p}(\Delta)[\alpha_{p}(\Delta) \otimes e_{p}] = f(\alpha(\Delta)).
	\end{equation*}
\end{proof}

	In conclusion, we prove the following proposition.
	\begin{prop}\label{prop: K is given by bounded chambers and resonant points}
		The kernel $ K $ of the map $ F $ in (\ref{eqn: definition of F}) is spanned by $ \alpha(\Delta)(\Delta \in C) $ and $ \alpha(p)^{\pm}(p \in R) $.
	\end{prop}

    \subsection{Explicit computation}\label{subsection: computation for global marking}

 The following lemma explicitly computes $ \lambda_{p}(\Delta) $.

\begin{lem}\label{lem: difference between extensions}
	There exist $ x_{0},y_{0} \in \RR $ such that $ x_{0} \neq 0 $ and $ (x(p)+x_{0},y(p)+y_{0}) \in \Delta $. Furthermore, 
	\begin{equation*}
		\lambda_{p}(\Delta) = \begin{cases}
					\prod\limits_{\substack{l \in \calA_{p} \\s(l)x_{0} > y_{0}}}m(l) & x_{0} < 0 \\
					1 & x_{0} > 0
		\end{cases}
	\end{equation*}
\end{lem}

\begin{proof}

The existence of $ x_{0}, y_{0} $ follows from the fact that $ \Delta $ is a nonempty open set in $ \RR^{2} $. Since $ \Delta $ is convex, by our choice of $ (x_{0},y_{0}) $, we have 
	\begin{equation*}
		\{(x(p)+stx_{0},y(p)+sty_{0}+s(1-t)i) \mid s \in (0,1],\ t \in [0,1] \} \subset  M(\calA).
	\end{equation*}

So we have the following homotopy in $ \widetilde{M} $:
\begin{equation*}
		\Phi \colon [0,1] \times [0,1] \rightarrow \widetilde{M},\ \Phi(s,t) = \begin{cases}
		    (x(p)+stx_{0},y(p)+sty_{0}+s(1-t)i) & s > 0 \\
            [tx_{0}:ty_{0}+(1-t)i] \in E_{p}^{\circ} & s = 0
		\end{cases}
	\end{equation*}

Denote by $ \gamma_{1}(t) = \Phi(t,0) $, $ \gamma_{2}(t) = \Phi(1,t) $, $ \gamma_{3}(t) = \Phi(1-t,1) $, and $ \sigma_{1}(t) = \Phi(0,t) $. Define
	\begin{equation*}
		\sigma_{2} \colon [0,1] \rightarrow E_{p}^{\circ},\ \sigma_{2}(t) = [(1-t):(1-t) \frac{y_{0}}{x_{0}}+ti].
	\end{equation*}

Recall that $ \gamma_{p} $ is exactly the segment $ \gamma_{1}^{-1} $ and the value of $ e(\overline{\Delta}) $ at the point $ [x_{0}:y_{0}] \in \alpha_{p}(\Delta) $ is obtained by extending $ e|_{\Delta} $ along the segment $ \gamma_{3} $. Thus $ \lambda_{p}(\Delta) $ is exactly the monodromy along the loop
	\begin{equation*}
		\begin{tikzcd}
				\left[x_{0}:y_{0}\right] \arrow{r}{\sigma_{0}} & \left[0:1\right] \arrow{r}{\gamma_{1}} & (x(p),y(p)+i) \arrow{r}{\gamma_{2}} & (x(p)+x_{0},y(p)+y_{0}) \arrow{r}{\gamma_{3}} & \left[x_{0}:y_{0}\right].
		\end{tikzcd}
	\end{equation*}

	Under the above homotopy $ \Phi $, the above loop is homotopic to the loop
	\begin{equation*}
		\begin{tikzcd}
			\left[x_{0}:y_{0}\right] \arrow{r}{\sigma_{2}} & \left[0:1\right] \arrow{r}{\sigma_{1}} & \left[x_{0}:y_{0}\right],
		\end{tikzcd}
	\end{equation*}
which is shown in Figure \ref{fig:3}.

 \begin{figure}[H]
 	\centering
 	\includegraphics[width=0.6\linewidth]{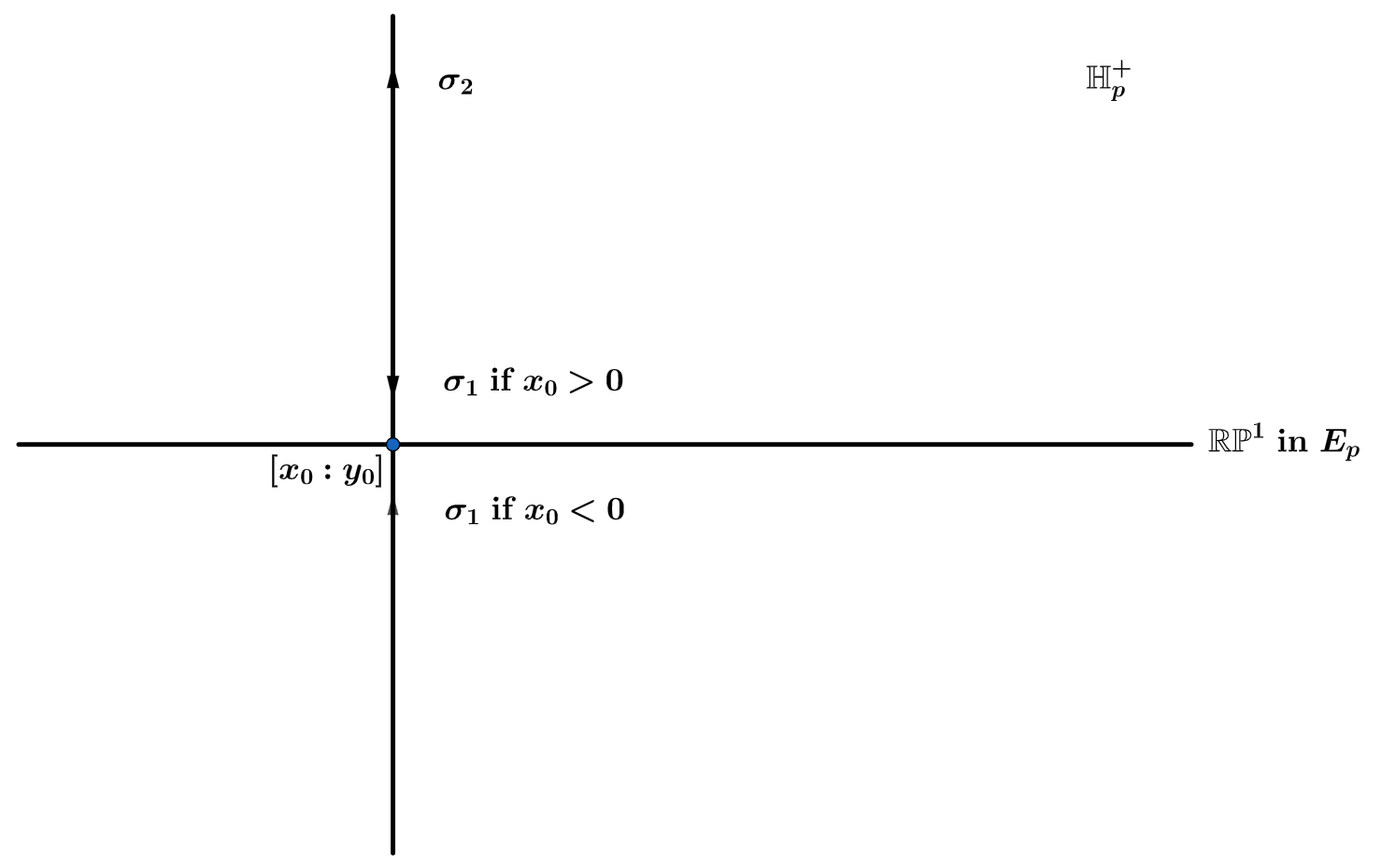}
 	\caption[1]{The loop through $ \left[x_{0}:y_{0}\right] $}
 	\label{fig:3}
 \end{figure}

When $ x_{0} > 0 $, this loop is contractible. So $ \lambda_{p}(\Delta) = 1 $.

When $ x_{0} < 0 $, This loop winds counterclockwise once around all points $ [1:s(l)] $ lying to the left of $ [x_{0}:y_{0}] $, which yields the equality
\begin{equation*}
   \lambda_{p}(\Delta) =  \prod\limits_{\substack{l \in \calA_{p} \\s(l)x_{0} > y_{0}}}m(l).
\end{equation*}
\end{proof}

\begin{rmk}\label{rmk: compare lambda with alpha}
     The value of $ \lambda_{p}(\Delta) $ is independent of the choice of $ (x_{0},y_{0}) $. More explicitly, suppose that $ \calA_{p} = \{ l_{1},\cdots,l_{k} \}$, where $ s(l_{1}) < \cdots < s(l_{k}) $. Then there are $ 2k $ chambers with vertex $ p $. Their corresponding $ \lambda_{p} $ are shown in Figure \ref{fig:4}

     \begin{figure}[H]
 	\centering
 	\includegraphics[width=0.5\linewidth]{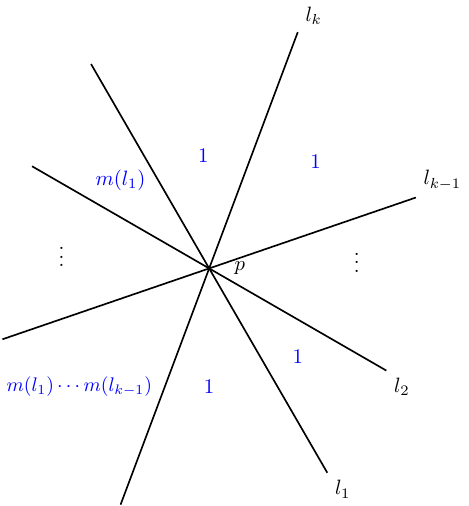}
 	\caption[1]{Coefficients $\lambda_{p}$}
 	\label{fig:4}
 \end{figure}

    In particular, comparing $ \lambda_{p}(\Delta) $ with the coefficients of $ \alpha(p)^{\pm} $ yields
    \begin{equation*}
        \sum\limits_{\substack{V(\Delta) \ni p \\ Q_{l_{1}}\mid_{\Delta} > 0}}\lambda_{p}(\Delta)\alpha_{p}(\Delta) = \alpha(p)^{+},\ \sum\limits_{\substack{V(\Delta) \ni p \\ Q_{l_{1}}\mid_{\Delta} < 0}}\lambda_{p}(\Delta)\alpha_{p}(\Delta) = \alpha(p)^{-}.
    \end{equation*}

     Similar equalities hold for $ l_{k} $.
\end{rmk}

		\begin{ex}\label{example: A3 arrangement}
		Let $ \calA $ be a reflection arrangement of type $ A_{3} $ as shown in Figure \ref{fig:5}. Let $ \LL $ be the complex rank-one local system with constant monodromy map $ \omega = e^{\frac{2\pi i}{3}} $. 

  \begin{figure}[H]
 	\centering
 	\includegraphics[width=0.7\linewidth]{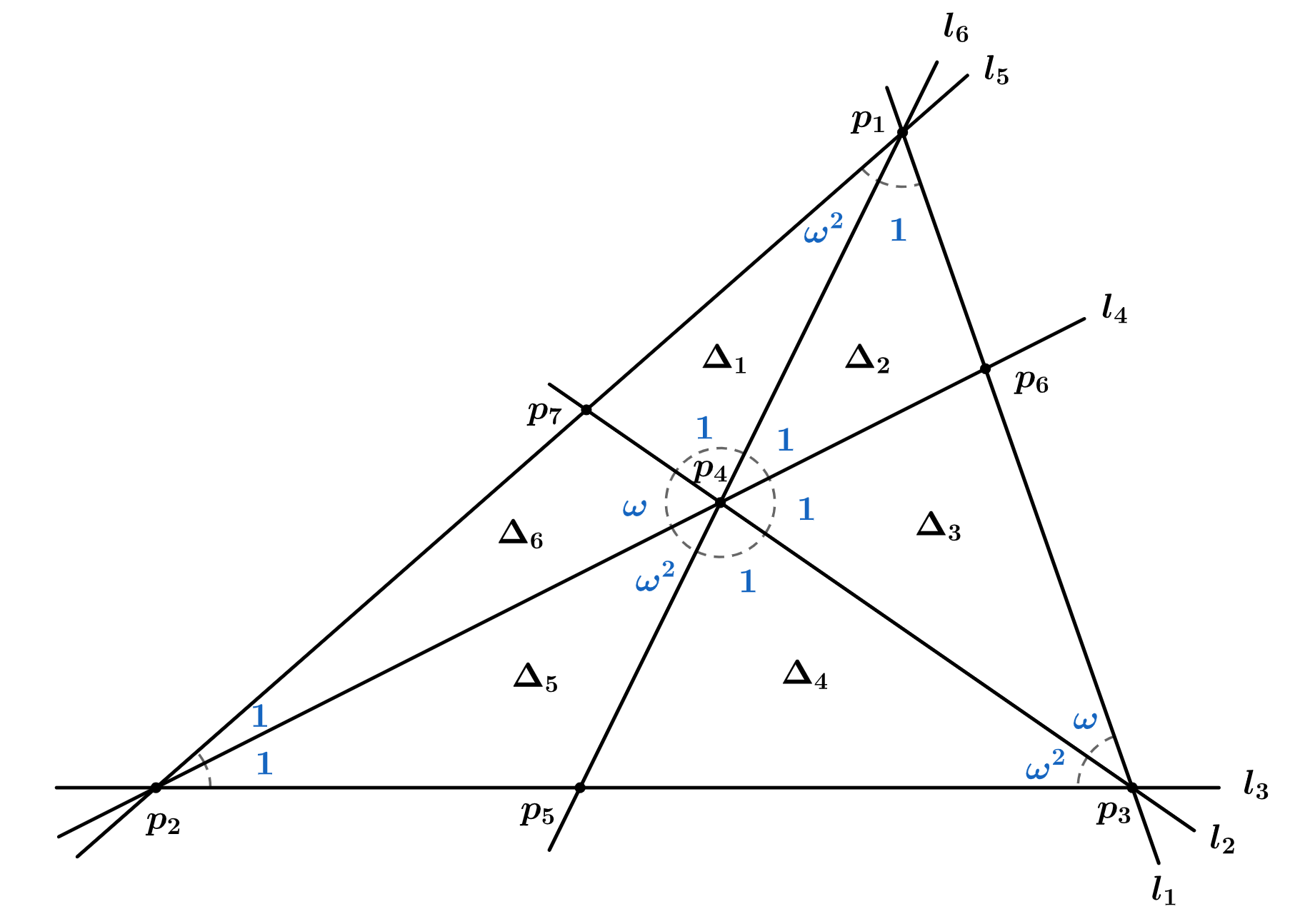}
 	\caption[1]{$ A_{3} $ arrangement}
 	\label{fig:5}
 \end{figure}
 
 	Then $ R = \{p_{1},p_{2},p_{3},p_{4}\} $ and $ A $ is a $ \CC $-vector space of dimension $ 12 $ spanned by angles
		\begin{equation*}
			(l_{1},l_{5}),\ 	(l_{5},l_{6}),\ 	(l_{6},l_{1}),\ 	(l_{3},l_{4}),\ 	(l_{4},l_{5}),\ 	(l_{5},l_{3}),\ 	(l_{1},l_{2}),\ 	(l_{2},l_{3}),\ 	(l_{3},l_{1}),\ 	(l_{2},l_{4}),\ 	(l_{4},l_{6}),\ 	(l_{6},l_{2}). 
		\end{equation*}
		
	Each resonant point gives two generators in $ K $. For example, at $ p_{1} $ we have
	\begin{equation*}
		\alpha(p_{1})^{+} = (l_{1},l_{5}) + (l_{5},l_{6}) + (l_{6},l_{1})
	\end{equation*}
	and
	\begin{equation*}
		\alpha(p_{1})^{-} = \omega(l_{1},l_{5}) + \omega^{2}(l_{5},l_{6}) + (l_{6},l_{1}).
	\end{equation*}
	
	Each bounded chamber gives one generator in $ K $. For example,
	for $ \Delta_{1} $ we have
	\begin{equation*}
		\alpha(\Delta_{1}) = \omega^{2}(l_{5},l_{6}) + (l_{6},l_{2}).
	\end{equation*}

	\setcounter{MaxMatrixCols}{14}
	So we have a $ 14 \times 12 $ matrix of rank $ 11 $:
	\begin{equation*}
	\begin{bmatrix}
	1	& 1 & 1 &  &  &  &  &  &  &  &  &  \\
	\omega	& \omega^{2} & 1 &  &  &  &  &  &  &  &  &  \\
		&  &  & 1 & 1 & 1 &  &  &  &  &  &  \\
		&  &  & \omega & \omega^{2} & 1 &  &  &  &  &  &  \\
		&  &  &  &  &  & 1 & 1 & 1 &  &  &  \\
		&  &  &  &  &  & \omega & \omega^{2} & 1 &  &  &  \\
		&  &  &  &  &  &  &  &  & 1 & 1 & 1 \\
		&  &  &  &  &  &  &  &  & \omega & \omega^{2} & 1 \\
		& \omega^{2} &  &  &  &  &  &  &  &  &  & 1 \\
		&  & 1 &  &  &  &  &  &  &  & 1 &  \\
		&  &  &  &  &  & \omega &  &  & 1 &  &  \\
		&  &  &  &  &  &  & \omega^{2} &  &  &  & 1 \\ 
		&  &  & 1 &  &  &  &  &  &  & \omega^{2} &  \\ 
		&  &  &  & 1 &  &  &  &  & \omega &  &  
	\end{bmatrix}.
	\end{equation*}
	
	Therefore, by Proposition \ref{prop: K is given by bounded chambers and resonant points} we have $ \dim_{\CC} K = 11 $. So $ h_{1}( M(\calA),\LL) = \dim_{\CC} A - \dim_{\CC} K = 1 $, which is compatible with classical results. Note that this gives a nontrivial example where the upper bound in our Theorem \ref{thm: main} is sharp.
	\end{ex}

	\section{Proof of Theorem \ref{thm: main}}\label{section: proof of main theorem}
	In this section, we prove Theorem \ref{thm: main}. Recall that $ \calA $ is a complexified real line arrangement with $ \# L_{2}(\calA) > 1 $, $ l_{0} $ is a given line in $ \calA $ and $ R_{0} $ is the set of resonant points on $ l_{0} $.
	
	Take a generic real line $ l_{\infty} $ in $ \CC\PP^{2} $ such that $ (l_{0},l_{\infty}) $ forms a sharp pair in $ \calA \cup \{l_{\infty}\} $. In other words, one of the two components of $ \RR\PP^{2}\setminus(l_{0}\cup l_{\infty}) $ contains no intersection points in $ L_{2}(\calA) $. Then, by taking $ l_{\infty} $ as the line at infinity and applying a further real coordinate transformation, we may assume that
	\begin{enumerate}
		\item 	$ l_{0} = \{y=0\} $,
		\item   The slopes $ s(l) $ of the lines $ l $ in $ \calA $ are distinct non-negative numbers, and 
		\item 	for any $ p \in L_{2}(\calA) $, its $ y $-coordinate is non-negative.
	\end{enumerate}

	  For each intersection point $ p $ on $ l_{0} $, denote by $ \calA'_{p} = \calA_{p} \setminus \{l_{0}\} $. If $ p $ further belongs to $ R_{0} $, we associate each line in $ \calA'_{p} $ with an element in $ A(p) $ as follows. Suppose that $ \calA'(p) = \{l_{1},\cdots,l_{k}\} $, where $ 0  < s(l_{1}) < \cdots <s(l_{k}) $. As shown in Figure \ref{fig:6}, define
	\begin{equation*}
		\alpha(l_{i}) = \sum\limits_{j=1}^{i}(l_{j-1},l_{j}) \in A(p).
	\end{equation*}
	
 \begin{figure}[H]
 	\centering
 	\includegraphics[width=0.5\linewidth]{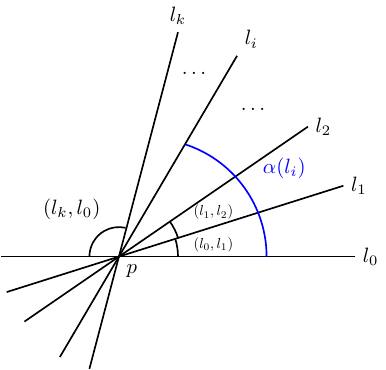}
 	\caption[1]{Definition of $ \alpha(l) $}
 	\label{fig:6}
 \end{figure}
    
	We further set
	\begin{equation*}
		\calA'  = \bigcup\limits_{p \in R_{0}} \calA'_{p}
	\end{equation*}
	 and
	\begin{equation*}
		A' = \bigoplus\limits_{l \in \calA'} \mathbb{K} \cdot \alpha(l) \subset A.
	\end{equation*}
	For convenience, we set $ \alpha(l) = 0 \in A $ if $ l \notin \calA' $.
	
	The following lemma reduces $ K $ to $  K \cap A' $.
	\begin{lem}\label{lem: reduction to a line}
		$ K+A' = A $.
	\end{lem}
	
	\begin{proof}
		For any $ p \in R_{0} $, recall that $ \alpha(p)^{+} = \sum\limits_{j=1}^{k}(l_{j-1},l_{j}) \in K $.  Since $ A(p) $ is spanned by $ \alpha(p)^{+} $ and $ \alpha(l)(l \in \calA'_{p}) $, we have $ A(p) \subset K+A' $. 
		
		Now it suffices to prove that for any $ q \in R\setminus R_{0} $, $ A(q) \subset K+A' $. We prove it by contradiction. Otherwise, we take $ q $ to be a point that minimizes $ y(q) $ among all resonant points such that $ A(q) \nsubseteq K+A' $.   Suppose that $ \calA_{q} = \{l_{1},\cdots,l_{k}\} $, where $ 0 < s(l_{1}) < \cdots <s(l_{k}) $. For each $ 1 \leq i \leq k-1 $, there exists a unique bounded chamber $ \Delta_{i} $ satisfying that $ \alpha_{q}(\Delta_{i}) = (l_{i},l_{i+1}) $ and $ 0 \leq y(p) < y(q) $ holds for all $ p $ in $ V(\Delta_{i}) $ other than $ q $. By minimality of $ y(q) $, we have 
		\begin{equation*}
			(l_{i},l_{i+1}) = \frac{1}{\lambda_{q}(\Delta_{i})}\left(\alpha(\Delta)-\sum\limits_{p \in V(\Delta_{i})\setminus \{q\}}\lambda_{p}(\Delta_{i})\alpha_{p}(\Delta_{i})\right) \in K+A'.
		\end{equation*}

		Furthermore, we have
		\begin{equation*}
			(l_{k},l_{1}) = \alpha(q)^{+} - \sum\limits_{i=1}^{k-1}(l_{i},l_{i+1}) \in K+A'
		\end{equation*}
		
		So $ A(q) \subset K+A' $, which contradicts the definition of $ q $.
	\end{proof}
	
	For each line $ l \in \calA \setminus \{l_{0}\} $, since $ \# L_{2}(\calA) > 1 $, there exists a unique intersection point $ q(l) $ on $ l $ but not on $ l_{0} $ with the minimum $ y $-coordinate. We call such a point $ q(l) $ a \textbf{neighbor} of $ l_{0} $. Note that different lines $ l $ may yield the same neighbor $ q(l) $. See Figure \ref{fig:7} for an example. The following lemma describes this phenomenon.
	\begin{lem}\label{lem: case of repeated connectable points}
		Let $ p_{1},p_{2} $ be two intersection points on $ l_{0} $ satisfying that $ x(p_{1}) < x(p_{2}) $. If there exist two lines $ l_{i} \in \calA'_{p_{i}}(i=1,2) $ such that $ q(l_{1}) = q(l_{2}) $, then
		\begin{enumerate}
			\item $ l_{1} $ has the minimum slope among the lines in $ \calA'_{p_{1}} $, 
			\item $ l_{2} $ has the maximum slope among the lines in $ \calA'_{p_{2}} $, and
			\item All intersection points between $ p_{1} $ and $ p_{2} $ are double points.
		\end{enumerate}
		In particular, for any neighbor $ q $ of $ l_{0} $, there are at most two lines $ l $ in $ \calA' $ such that $ q = q(l) $. 
	\end{lem}

     \begin{figure}[H]
 	\centering
 	\includegraphics[width=0.7\linewidth]{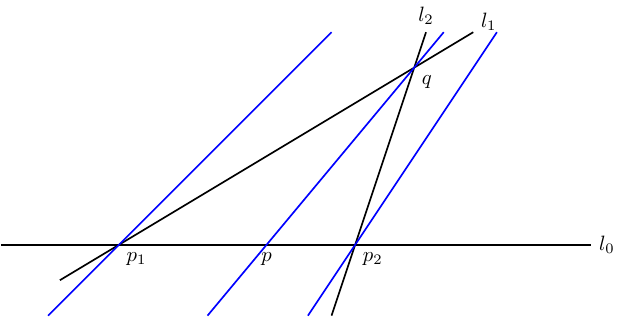}
 	\caption[1]{The case when $ q(l_{1}) = q(l_{2}) $}
 	\label{fig:7}
 \end{figure}
    
	\begin{proof}
	Denote by $ q = q(l_{1}) = q(l_{2}) $. Since all lines through $ p $ with slopes in $ (0,s(l_{1})) $ intersect with the segment $ (p_{2},q) $, we have $ l_{1} $ has the minimum slope among the lines in $ \calA'_{p_{1}} $. Similarly we have $ l_{2} $ has the maximum slope among the lines in $ \calA'_{p_{2}} $. Furthermore, for any point $ p \in L_{2}(\calA) $ between $ p_{1} $ and $ p_{2} $ and any line $ l \in \calA'_{p} $, $ l $ passes through $ q $ since there is no intersection point in $ (p_{1},q) $ and $ (p_{2},q) $. So $ p $ is a double point.
	\end{proof}

	The following lemma shows that each neighbour of $ l_{0} $ provides an element in $ K \cap A' $.
	\begin{lem}\label{lem: connectable points provide elements}
		Let $ q $ be a neighbor of $ l_{0} $. Suppose that $ \calA_{q} = \{l_{1},\cdots,l_{k}\} $, where $ 0 < s(l_{1}) < \cdots <s(l_{k}) $. Denote $ m(l_{i}) $ by $ m_{i} $.
		\begin{enumerate}
			\item If $ q $ is resonant, then
				 \begin{equation*}
					 \sum\limits_{i=1}^{k} \frac{m_{i}-1}{m_{1}\cdots m_{i}}\alpha(l_{i}) \in K \cap A'.
				\end{equation*}
			\item If $ q $ is not resonant, then for any $ 1 \leq i < j \leq k $, $ \alpha(l_{i})-\alpha(l_{j}) \in K \cap A' $.
		\end{enumerate}
	\end{lem}
	
	\begin{proof}
		  Suppose that $ q = q(l_{k_{0}}) $ for some $ 1 \leq k_{0} \leq k $. For $ 1 \leq i \leq k $, denote by $ p_{i} $ the intersection of $ l_{i} $ and $ l_{0} $. For $ 1 \leq i \leq k-1 $, let $ T_{i} $ be the closed triangle bounded by $ l_{0} $, $ l_{i} $ and $ l_{i+1} $. Denote by $ C_{i} $ the set consisting of bounded chambers inside $ T_{i} $ and $ V_{i} $ the set consisting of resonant points in $ T_{i} \setminus \{q,p_{i},p_{i+1}\} $. We claim that for any $ p \in V_{i} $,
			\begin{equation*}
				 \pi_{p}(\sum\limits_{\Delta \in C_{i}}\alpha(\Delta))  \in K,
			\end{equation*}
		where $ \pi_{p} \colon A \rightarrow A(p) $ is the projection map. We verify this claim by considering the position of $ p $. Suppose that $ \calA_{p} = \{l'_{1},\cdots,l'_{n}\} $, where $ 0 < s(l'_{1}) < \cdots <s(l'_{n}) $. 
		 
		 \begin{enumerate}
		 	\item When $ p $ lies in the interior of $ T_{i} $, we have
		 	\begin{equation*}
		 		\pi_{p}(\sum\limits_{\Delta \in C_{i}}\alpha(\Delta))  =  \sum\limits_{j=1}^{n}(m(l'_{1})\cdots m(l'_{j})+1)(l'_{j},l'_{j+1})  =  \alpha(p)^{+}+\alpha(p)^{-} \in K.
		 	\end{equation*}
		 	
		 	\item When $ p $ lies inside the segment $ (p_{i},p_{i+1}) $, we have
		 	\begin{equation*}
		 		\pi_{p}(\sum\limits_{\Delta \in C_{i}}\alpha(\Delta))   =  \sum\limits_{j=1}^{n}(l'_{j},l'_{j+1}) 
		 		 =  \alpha(p)^{+} \in K.
		 	\end{equation*} 
		 	
		 	\item When $ p $ lies inside the segment $ (p_{i},q) $, since  $ q = q(l_{k_{0}}) $, we have $ i \neq k_{0} $. 
		 	
		 	If $ i < k_{0} $, as shown in Figure \ref{fig:8}, the fact that all lines through $ p $ with slopes in $ [0,m_{i}) $ intersect with the segment $ (p_{k_{0}},q) $ forces $ l_{i} = l'_{1} $. By Remark \ref{rmk: compare lambda with alpha}, we have
		 	\begin{equation*}
		 		\pi_{p}(\sum\limits_{\Delta \in C_{i}}\alpha(\Delta))   =   \sum\limits_{j=1}^{n}m(l'_{1})\cdots m(l'_{j})(l'_{j},l'_{j+1}) = \alpha(p)^{-} \in K.
		 	\end{equation*}

             \begin{figure}[H]
 	\centering
 	\includegraphics[width=0.7\linewidth]{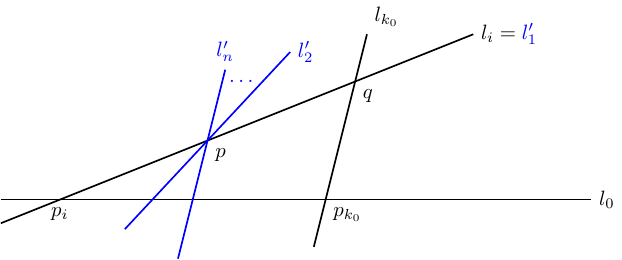}
 	\caption[1]{$ l_{i} $ has the minimum slope among the lines in $ \calA_{p} $}
 	\label{fig:8}
 \end{figure}
            
		 	Similarly, if $ i>k_{0} $, we have $ l_{i} = l'_{n} $ and
		 	\begin{equation*}
		 		\pi_{p}(\sum\limits_{\Delta \in C_{i}}\alpha(\Delta))  =  \sum\limits_{j=1}^{n}(l'_{j},l'_{j+1})  =  \alpha(p)^{+} \in K.
		 	\end{equation*}
		 \item When $ p $ lies inside the segment $ (p_{i+1},q) $, since  $ q = q(l_{k_{0}}) $, we have $ i+1 \neq k_{0} $. 
		 
		 If $ i+1 < k_{0} $, we have $ l_{i+1} = l'_{1} $ and
		 \begin{equation*}
		 	\pi_{p}(\sum\limits_{\Delta \in C_{i}}\alpha(\Delta))  =  \sum\limits_{j=1}^{n}(l'_{j},l'_{j+1})  =  \alpha(p)^{+} \in K.
		 \end{equation*}

		 Similarly, if $ i>k_{0} $, we have $ l_{i+1} = l'_{n} $ and
		 \begin{equation*}
		 	\pi_{p}(\sum\limits_{\Delta \in C_{i}}\alpha(\Delta))   =   \sum\limits_{j=1}^{n}m(l'_{1})\cdots m(l'_{j})(l'_{j},l'_{j+1}) = \alpha(p)^{-} \in K.
		 \end{equation*}
		 \end{enumerate}
		 
		 
		 As a result, taking 
		 \begin{equation*}
		 	\beta_{i}(q) = \sum\limits_{\Delta \in C_{i}}\alpha(\Delta) - \sum\limits_{p \in V_{i}}\pi_{p}(\sum\limits_{\Delta \in C_{i}}\alpha(\Delta)),
		 \end{equation*}
		 we have $ \beta_{i}(q) \in K $.

	  When $ q $ is resonant, for each $ 1 \leq i \leq k-1 $, we have
	  \begin{equation*}
	  	 \beta_{i}(q) = m_{1}\cdots m_{i}(l_{i+1},l_{i})+\alpha(l_{i})-\alpha(l_{i+1}).
	  \end{equation*}
  
  	 Therefore,
	  \begin{equation*}
	  	\sum\limits_{i=1}^{k} \frac{m_{i}-1}{m_{1}\cdots m_{i}}\alpha(l_{i}) = \alpha(q)^{+} - \alpha(q)^{-} +  \sum\limits_{i=1}^{k-1} \frac{m_{1}\cdots m_{i}-1}{m_{1}\cdots m_{i}}\beta_{i}(q)  \in K.
	  \end{equation*}

	When $ q $ is not resonant, for each $ 1 \leq i \leq k-1 $, we have
	\begin{equation*}
		\alpha(l_{i})-\alpha(l_{i+1}) = \beta_{i}(q) \in K,
	\end{equation*}
	which implies the conclusion holds.
	\end{proof}
	
	\begin{cor}\label{cor: definition of beta}
		For any neighbor $ q $ of $ l_{0} $, there exists a $ \mathbb{K} $-linear combination $ \beta(q) $ of $ \alpha(l)(l \in \calA_{q} \cap \calA') $ in $ A' $ such that
		\begin{enumerate}
			\item $ \beta(q) \in K \cap A' $, and
			\item for any $ l \in \calA_{q} \cap \calA' $ such that $ q(l) = q $, the coefficient of $ \alpha(l) $ in $ \beta(q) $ is nonzero.
		\end{enumerate}
	\end{cor}
	
	\begin{proof}
		Suppose that $ \calA_{q} = \{l_{1},\cdots,l_{k}\} $, where $ 0 < s(l_{1}) < \cdots <s(l_{k}) $. Denote $ m(l_{i}) $ by $ m_{i} $.
		
		When $ q $ is resonant, take
		\begin{equation*}
			\beta(q) = \sum\limits_{i=1}^{k} \frac{m_{i}-1}{m_{1}\cdots m_{i}}\alpha(l_{i}) = \sum\limits_{l_{i} \in \calA'} \frac{m_{i}-1}{m_{1}\cdots m_{i}}\alpha(l_{i}). 
		\end{equation*} 
		
		Then $ \beta(q) \in K \cap A' $ by Lemma \ref{lem: connectable points provide elements}, and its coefficients are all nonzero since $ m(l) \neq 1 $ for any $ l \in \calA $. 
		
		When $ q $ is not resonant, by Lemma \ref{lem: case of repeated connectable points}, there exist $ 1 \leq i_{1} < i_{2} \leq k $ such that
		\begin{equation*}
			\{l \in \calA_{q} \cap \calA' \mid q(l) = q\} \subseteq \{l_{i_{1}},l_{i_{2}}\}.
		\end{equation*}
		
		 Then by Lemma \ref{lem: connectable points provide elements}, we have $ \beta(q) = \alpha(l_{i_{1}}) - \alpha(l_{i_{2}}) $ meets the requirements. 
	\end{proof}

Now we are ready to prove our main theorem. 
\begin{proof}[Proof of Theorem \ref{thm: main}]
	If $ R_{0} = \emptyset $, then $ A' = 0 $ and $ K = A $. Thus $ H_{1}(M(\calA),\LL) = 0 $. 
	
	Henceforth assume that $ \# R_{0} > 1 $. Define $ N = \{q(l) \mid l \in \calA'\} $. We first prove that $ \beta(q)(q \in N) $ are linearly independent in $ K \cap A' $, where $ \beta(q) $ is defined in Corollary \ref{cor: definition of beta}. 
	
	Otherwise, there exist points $ q_{1},\cdots,q_{k} \in N $ and nonzero coefficients $ \lambda_{1},\cdots,\lambda_{k} \in \mathbb{K} $, such that the linear combination
	\begin{equation*}
		\beta = \sum\limits_{i=1}^{k}\lambda_{i}\beta(q_{i}) = 0.
	\end{equation*}
	Without loss of generality, we may assume that $ q_{1} $ has the maximum $ y $-coordinate among $ q_{1},\cdots,q_{k} $. Suppose that $ q_{1} = q(l_{1}) $ for $ l_{1} \in \calA' $. Since $ l_{1} $ contains no points in $ L_{2}(\calA) $ with positive $ y $-coordinate less than $ y(q_{1}) $, we have $ l_{1} \notin \calA_{q_{i}} $ holds for any $ i \geq 2 $. So the coefficient of $ \alpha(l_{1}) $ in $ \beta $ is equal to the one in $ \lambda_{1}\beta(q_{1}) $, which is nonzero by Corollary \ref{cor: definition of beta}. This contradicts the assumption $ \beta = 0 $. So $ \beta(q)(q \in N) $ are linearly independent in $ K \cap A' $.
	
   Then we estimate the size of $ N $. Suppose that $ R_{0} = \{p_{1},\cdots,p_{m}\} $, where $ m = \# R_{0} > 0 $ and $ x(p_{1}) < \cdots < x(p_{m}) $. Denote by $ l_{i}^{+} $ (resp. $ l_{i}^{-} $) the line in $ \calA'_{p_{i}} $ with the maximum (resp. minimum) slope. For any $ 1 \leq i \leq m $, since $ \mult(p_{i}) > 2 $, we have $ l_{i}^{+} \neq l_{i}^{-} $. So by Lemma \ref{lem: case of repeated connectable points}, for any $ q \in N $,
   \begin{equation*}
   	\# \{l \in \calA' \mid q(l) = q\} \leq 2.
   \end{equation*}
   Furthermore, when the equality holds, there exists some index $ 1 \leq i \leq m-1 $ such that $ q = q(l_{i}^{-}) = q(l_{i+1}^{+}) $. Therefore,
   \begin{equation*}
   	\# N \geq \# \calA' - (m-1) = \# \calA' - \# R_{0}+1.
   \end{equation*}
   
   	So combining Proposition \ref{prop: K is given by bounded chambers and resonant points} and Lemma \ref{lem: reduction to a line}, we have
   \begin{equation*}
   	h_{1}( M(\calA),\LL) = \dim A - \dim K = \dim A' - \dim K\cap A' \leq \# \calA' - \# N \leq  \# R_{0} - 1.
   \end{equation*}
\end{proof}

\section{The case of line arrangements containing a sharp pair}\label{section: sharp pair}

In this section, we discuss the case where $ \calA $ is a complexified real line arrangement containing a sharp pair $ (l_{0},l'_{0}) $. In other words, one of the two components of $ \RR\PP^{2}\setminus(l_{0}\cup l'_{0}) $ contains no intersection points in $ L_{2}(\calA) $. Denote by $ p_{0}$ the intersection point of $ l_{0} $ and $ l'_{0} $.

Take a generic real line $ l_{\infty} $ in $ \CC\PP^{2} $ such that both $ (l_{0},l_{\infty}) $ and $ (l_{0},l'_{0}) $ are sharp pairs in $ \calA \cup \{l_{\infty}\} $. Then, by taking $ l_{\infty} $ as the line at infinity and applying a further real coordinate transformation, we may assume that
\begin{enumerate}
			\item 	$ l_{0} = \{y=0\} $ and $ l'_{0} = \{y = x\} $,
			\item   The slopes $ s(l) $ of the lines $ l $ in $ \calA $ are distinct numbers in $ [0,1] $, and 
			\item 	for any $ p \in L_{2}(\calA) $, either $ x(p) \leq 0 $ and $ y(p) = 0 $, or $ x(p) \geq y(p) > 0 $.		
\end{enumerate}
Under the above assumptions, we can keep the notation from Section \ref{section: proof of main theorem}. Furthermore, for any $ l \in \calA \setminus \calA_{p_{0}} $, the point $ q(l) $ is exactly the intersection point of $ l $ and $ l'_{0} $. This property directly leads to the following lemma.
	\begin{lem}\label{lem: all elements are scalar multiples of sharp pair}
		 For any $ l \in \calA \setminus \calA_{p_{0}} $, $ [\alpha(l)] $ is a scalar multiple of $ [\alpha(l'_{0})] $ in $ A'/K\cap A' $.
	\end{lem}
	
	\begin{proof}
		It suffices to prove that for every intersection point $ p $ on $ l_{0} $ other than $ p_{0} $,  $ [\alpha(l)] $ is a scalar multiple of $ [\alpha(l'_{0})] $ for any $ l \in \calA'_{p} $. 
		
		Without loss of generality, we may assume that $ p $ is resonant. Suppose that $ \calA'(p) = \{l_{1},\cdots,l_{k}\} $, where $ 0 < s(l_{1}) < \cdots < s(l_{k}) $. Denote by $ q_{i} $ the intersection point of $ l_{i} $ and $ l'_{0} $. Suppose that $ \calA_{q_{1}} = \{l'_{0},l'_{1},\cdots,l'_{n}\} $, where $ 1 = s(l'_{0}) > \cdots > s(l'_{n}) $. Denote by $ p_{j} $ the intersection points of $ l'_{j} $ and $ l_{0} $. See Figure \ref{fig:9}.

         \begin{figure}[H]
 	\centering
 	\includegraphics[width=0.5\linewidth]{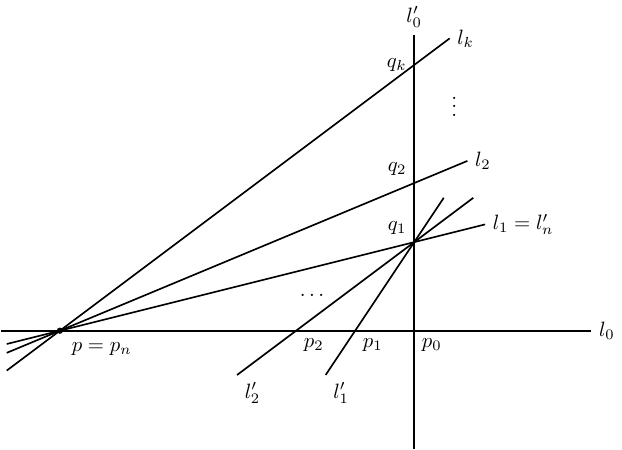}
 	\caption[1]{Picture for Lemma \ref{lem: all elements are scalar multiples of sharp pair}}
 	\label{fig:9}
 \end{figure}

		Since there is no intersection points in the region $ \{ (x,y) \in \RR^{2} \mid y>x,y>0\} $, we have $ q_{i} = q(l_{i}) $, $ l_{1} = l'_{n} $, and $ p_{j}(2 \leq j \leq n-1) $ and $ q_{i}(2 \leq i \leq k-1) $ are all double points. Thus, by applying Lemma \ref{lem: connectable points provide elements} to the point $ q_{i}(1\leq i \leq k-1) $ we can express $ [\alpha(l_{i})] $ as linear combinations of $ [\alpha(l'_{1})] $ and $ [\alpha(l'_{0})] $ as follows.

	For $ l_{i}(2 \leq i \leq k-1) $, since $ q_{i} $ is a double point, we have $ [\alpha(l_{i})] = [\alpha(l'_{0})] $. For $ l_{1} $, if $ q_{1} $ is not resonant, we have $ 	[\alpha(l_{1})] = [\alpha(l'_{0})]  $. Otherwise, since $ l'_{j} \notin \calA' $ for any $ 2 \leq j \leq n-1 $, we have
			\begin{equation*}
				[\alpha(l_{1})] = 
					 -\frac{(m(l'_{1})-1)m(l_{1})m(l'_{0})}{m(l_{1})-1}[\alpha(l'_{1})]-\frac{(m(l'_{0})-1)m(l_{1})}{m(l_{1})-1}[\alpha(l'_{0})].
			\end{equation*}	 
			 
		Furthermore, for $ l_{k} $, we have
			\begin{equation*}
				(m(l_{0})\cdots m(l_{k-1})-1)\alpha(l_{k}) - \sum\limits_{i=1}^{k-1}(m(l_{i})-1)m(l_{0})\cdots m(l_{i-1})\alpha(l_{i}) =  \alpha(p)^{-}-\alpha(p)^{+} \in K.
			\end{equation*}
		
		Combining the above computations, we have
		\begin{equation*}
			[\alpha(l_{k})] = \frac{(m(l_{1})-1)m(l_{0})m(l_{k})}{1-m(l_{k})}[\alpha(l_{1})]+\frac{1-m(l_{0})m(l_{1})m(l_{k})}{1-m(l_{k})}[\alpha(l'_{0})]
		\end{equation*}
	is also a linear combination of $ [\alpha(l'_{1})] $ and $ [\alpha(l'_{0})] $.
	
	Note that when $ q_{1} $ is not resonant, the above formulae already imply that $ [\alpha(l_{i})](1 \leq i \leq k) $ are all scalar multiples of $ [\alpha(l'_{0})] $. When $ q_{1} $ is resonant, we have $ n \geq 2 $ and $ 0 > x(p_{1}) > x(p) $. Thus the conclusion follows easily by induction on $ x(p) $. 
	\end{proof}
	
	\begin{cor}[Theorem \ref{thm: main for sharp pair}]\label{cor: bound for sharp pair}
		Under the above assumptions, $ A'/K \cap A' $ is generated by $ [\alpha(l'_{0})] $. In particular, $ h_{1}(M(\calA),\LL) \leq 1 $. 
	\end{cor}
	
	\begin{proof}
		By Lemma \ref{lem: all elements are scalar multiples of sharp pair}, it suffices to prove that for any $ l \in \calA'_{p_{0}} \setminus \{l'_{0}\} $, $ [\alpha(l)] $ is a scalar multiple of $ [\alpha(l'_{0})] $. Denote by $ q = q(l) $. Since $ (\calA_{q}\setminus\{l\}) \subset (\calA \setminus \calA_{p_{0}}) $, for any $ l' \in \calA_{q}\setminus\{l\} $, $ [\alpha(l')] $ is a scalar multiple of $ [\alpha(l'_{0})] $. So by Lemma \ref{lem: connectable points provide elements}, $ [\alpha(l)] $ is a linear combination of $ [\alpha(l')](l' \in \calA_{q}\setminus\{l\}) $, which is still a scalar multiple of $ [\alpha(l'_{0})] $.

	\end{proof}
	
	Now we are ready to prove Theorem \ref{thm: odd monodromy}.
	 
	 \begin{proof}[Proof of Theorem \ref{thm: odd monodromy}]
	 	
	 	 Without loss of generality, we may assume $ \zeta \neq 1 $. Then by Corollary \ref{cor: bound for sharp pair}, we have $ H_{1}(M(\calA),\LL) \simeq A'/K \cap A' = \mathbb{K} \cdot [\alpha(l'_{0})] $ and $ [\alpha(l'_{0})] \neq 0 $.
	 	
	 	Construct a sequence of lines in $ \calA $ as follows. Take $ p_{1} \in R_{0} $ with the maximum negative $ x $-coordinate. Let $ l_{1} $ (resp. $ l_{2} $) be the line in $ \calA'_{p_{1}} $ with the minimum (resp. maximum) slope. Denote by $ q_{i} $ the intersection point of $ l_{i} $ and $ l'_{0} $.
	 	
	 	Assume that the lines $ l_{1},\cdots,l_{2k} $, the points $ p_{1},\cdots,p_{k},q_{1},\cdots,q_{k+1} $ have been taken. If $ q_{k+1} $ is not resonant, stop. Otherwise, let $ l_{2k+1} $ be the line in $ \calA_{q_{k+1}} $ with the minimum slope and let $ p_{k+1} $ be the intersection point of $ l_{2k+1} $ and $ l_{0} $. If $ p_{k+1} $ is not resonant, stop. Otherwise, let $ l_{2k+2} $ be the line in $ \calA'_{p_{k+1}} $ with the maximum slope and let $ q_{k+2} $ be the intersection point of $ l_{2k+2} $ and $ l'_{0} $. Then repeat this step with $ k $ replaced by $ k+1 $.
	 	
	 	This process necessarily terminates since $ \calA $ is a finite set. Eventually we obtain a sequence of lines $ l_{1},\cdots,l_{n} $ in $ \calA $. In particular,
	 	 \begin{equation*}
	 		[\alpha(l_{n})] = \begin{cases}
	 			0 & \text{If }2 \nmid n \\
	 			[\alpha(l'_{0})] & \text{Otherwise}
	 		\end{cases}
	 	\end{equation*}

	 	As we computed in the proof of Lemma \ref{lem: all elements are scalar multiples of sharp pair}, we have the following recurrence formulae:
	 	\begin{equation*}
	 		[\alpha(l_{2k})] = -\zeta^{2}[\alpha(l_{2k-1})]+(1+\zeta+\zeta^{2})[\alpha(l'_{0})]
	 	\end{equation*}
	 	and
 		\begin{equation*}
 		[\alpha(l_{2k+1})] = 
 		-\zeta^{2}[\alpha(l'_{2k})]-\zeta[\alpha(l'_{0})].
	 	\end{equation*}
 
 		The initial value $ [\alpha(l_{1})] $ depends on whether $ q_{1} $ is a resonant point, and this ultimately leads to different equations at $ l_{n} $ as shown below.
 		\begin{enumerate}
 			\item If $ q_{1} $ is resonant, the maximality of $ x(p_{1}) $ implies that
 			\begin{equation*}
 				[\alpha(l_{1})] = -\zeta[\alpha(l'_{0})].
 			\end{equation*}
 			
 			Thus, by recursion, we have
 			\begin{equation*}
 				[\alpha(l_{2k})] = (\sum\limits_{i=0}^{4k-1}\zeta^{i})[\alpha(l'_{0})]
 			\end{equation*}
 			and
 			\begin{equation*}
 				[\alpha(l_{2k+1})] = -(\sum\limits_{i=1}^{4k+1}\zeta^{i})[\alpha(l'_{0})].
 			\end{equation*}
 			
 			In particular, at $ l_{n} $, we have
 			\begin{equation*}
 				(\zeta^{2n-1}-1)[\alpha(l'_{0})] = (\zeta-1)(\sum\limits_{i=0}^{2n-2}\zeta^{i})[\alpha(l'_{0})] = 0,
 			\end{equation*}
 			which implies that $ \zeta^{2n-1} = 1 $.
 			\item If $ q_{1} $ is not resonant, then $ 	[\alpha(l_{1})] = [\alpha(l'_{0})] $. Similarly we have, for $ k \geq 1 $,
 			\begin{equation*}
 				[\alpha(l_{2k})] = (\sum\limits_{i=0}^{4k-3}\zeta^{i})[\alpha(l'_{0})]
 			\end{equation*}
 			and
 			\begin{equation*}
 				[\alpha(l_{2k+1})] = -(\sum\limits_{i=1}^{4k-1}\zeta^{i})[\alpha(l'_{0})].
 			\end{equation*}
 			
 			In particular, at $ l_{n} $, we have
 			\begin{equation*}
 				(\zeta^{2n-3}-1)[\alpha(l'_{0})] = (\zeta-1)(\sum\limits_{i=0}^{2n-4}\zeta^{i})[\alpha(l'_{0})] = 0,
 			\end{equation*}
 			which implies that $ \zeta^{2n-3} = 1 $.
 		\end{enumerate}

 		In both cases there exists an odd number $ d $ such that $ \zeta^{d} = 1 $.
\end{proof}
	 
	 \bibliographystyle{plain}
	\bibliography{reference}
\end{document}